\documentclass[11pt]{article}
\usepackage[margin=1in]{geometry} 
\usepackage{amssymb,amsfonts,amsmath,bbm,mathrsfs,stmaryrd}
\usepackage{xcolor}
\usepackage{url}

\usepackage{enumerate}

\usepackage[colorlinks,
             linkcolor=black!75!red,
             citecolor=blue,
             pdftitle={},
             pdfproducer={pdfLaTeX},
             pdfpagemode=None,
             bookmarksopen=true
             bookmarksnumbered=true]{hyperref}

\usepackage{tikz}
\usetikzlibrary{arrows,calc,decorations.pathreplacing,decorations.markings,intersections,shapes.geometric,through,fit,shapes.symbols,positioning,decorations.pathmorphing}

\usepackage{braket}

\usepackage[amsmath,thmmarks,hyperref]{ntheorem}
\usepackage{cleveref}

\creflabelformat{enumi}{#2(#1)#3}

\crefname{section}{Section}{Sections}
\crefformat{section}{#2Section~#1#3} 
\Crefformat{section}{#2Section~#1#3} 

\crefname{subsection}{\S}{\S\S}
\AtBeginDocument{%
  \crefformat{subsection}{#2\S#1#3}%
  \Crefformat{subsection}{#2\S#1#3}%
}

\theoremstyle{plain}

\newtheorem{lemma}{Lemma}[section]

\newtheorem{corollary}[lemma]{Corollary}
\newtheorem{theorem}[lemma]{Theorem}
\newtheorem{conjecture}[lemma]{Conjecture}

\theoremstyle{nonumberplain}

\theoremstyle{plain}
\theorembodyfont{\upshape}
\theoremsymbol{\ensuremath{\blacklozenge}}

\newtheorem{definition}[lemma]{Definition}

\newtheorem{remark}[lemma]{Remark}

\crefname{definition}{definition}{definitions}
\crefformat{definition}{#2definition~#1#3} 
\Crefformat{definition}{#2Definition~#1#3} 

\crefname{ex}{example}{examples}
\crefformat{example}{#2example~#1#3} 
\Crefformat{example}{#2Example~#1#3} 

\crefname{remark}{remark}{remarks}
\crefformat{remark}{#2remark~#1#3} 
\Crefformat{remark}{#2Remark~#1#3} 

\crefname{convention}{convention}{conventions}
\crefformat{convention}{#2convention~#1#3} 
\Crefformat{convention}{#2Convention~#1#3}

\crefname{lemma}{lemma}{lemmas}
\crefformat{lemma}{#2lemma~#1#3} 
\Crefformat{lemma}{#2Lemma~#1#3} 

\crefname{proposition}{proposition}{propositions}
\crefformat{proposition}{#2proposition~#1#3} 
\Crefformat{proposition}{#2Proposition~#1#3} 

\crefname{corollary}{corollary}{corollaries}
\crefformat{corollary}{#2corollary~#1#3} 
\Crefformat{corollary}{#2Corollary~#1#3} 

\crefname{theorem}{theorem}{theorems}
\crefformat{theorem}{#2theorem~#1#3} 
\Crefformat{theorem}{#2Theorem~#1#3} 

\crefname{conjecture}{conjecture}{conjectures}
\crefformat{conjecture}{#2conjecture~#1#3} 
\Crefformat{conjecture}{#2Conjecture~#1#3} 

\crefname{enumi}{}{}
\crefformat{enumi}{(#2#1#3)}
\Crefformat{enumi}{(#2#1#3)}

\crefname{assumption}{assumption}{Assumptions}
\crefformat{assumption}{#2assumption~#1#3} 
\Crefformat{assumption}{#2Assumption~#1#3} 

\crefname{equation}{}{}
\crefformat{equation}{(#2#1#3)} 
\Crefformat{equation}{(#2#1#3)}

\numberwithin{equation}{section}

\theoremstyle{nonumberplain}
\theoremsymbol{\ensuremath{\blacksquare}}

\newtheorem{proof}{Proof}

\newcommand\bN{{\mathbb N}}

\newcommand\bR{{\mathbb R}}

\newcommand\bZ{{\mathbb Z}}

\newcommand\cM{{\mathcal M}}

\newcommand\cS{{\mathcal S}}

\newcommand{\qedhere}{\mbox{}\hfill\ensuremath{\blacksquare}}

\title{Tree-optimized directed graphs}
\author{Alexandru Chirvasitu}

\begin{document}

\date{}

\newcommand{\Addresses}{{
  \bigskip
  \footnotesize

  \textsc{Department of Mathematics, University at Buffalo, Buffalo,
    NY 14260-2900, USA}\par\nopagebreak \textit{E-mail address}:
  \texttt{achirvas@buffalo.edu}

}}

\maketitle

\begin{abstract}
  For an additive submonoid $\mathcal{M}$ of $\bR_{\ge 0}$, the weight of an $\mathcal{M}$-labeled directed graph is the sum of all of its edge labels, while the content is the product of the labels. Having fixed $\mathcal{M}$ and a directed tree $E$, we prove a general result on the shape of directed $\mathcal{M}$-labeled graphs $\Gamma$ of weight $N\in \mathcal{M}$ maximizing the sum of the contents of all copies $E\subset \Gamma$.

  This specializes to recover a result of Hajac and Tobolski on the maximal number of length-$k$ paths in a directed acyclic graph. It also applies to prove a conjecture by the same authors on the maximal sum of entries of $A^k$ for a nilpotent $\bR_{\ge 0}$-valued square matrix $A$ whose entries add up to $N$. Finally, we apply the same techniques to obtain the maximal number of stars with $a$ arms in a directed graph with $N$ edges.
\end{abstract}


\vspace{.5cm}



\section*{Introduction}

This note is motivated by \cite[Theorem 1.10]{ht-gr} and various ramifications thereof. The result in question gives a sharp upper bound for the number of length-$k$ paths in a directed acyclic graph (henceforth DAG, for short) with $N$ edges:

\begin{theorem}\label{th:init}
  Let $k$ and $N$ be positive integer, with $N=kq+r$ be the decomposition of $N$ modulo $k$. Then, a DAG with $N$ edges contains at most $(q+1)^r q^{k-r}$ directed paths of length $k$. 
\end{theorem}

There is an alternative way to state the result, that is perhaps more conceptually expressive:

\begin{corollary}\label{cor:init}
  Let $N$ and $k$ be two positive integers. The following numbers are then equal:
  \begin{enumerate}[(a)]
  \item\label{item:initlen} the maximal number of length-$k$ paths in an $N$-edge DAG;
  \item\label{item:initnr} the maximal product of $k$ non-negative integers with sum $N$. 
  \end{enumerate}
\end{corollary}

The fact that the maximum in point \Cref{item:initnr} is achieved for the ``best-balanced'' $k$-tuple
\begin{equation*}
  q+1,\ q+1,\ \cdots,\ q+1,\ q,\ q,\ \cdots,\ q
\end{equation*}
of positive integers can be seen easily, by noting for instance that if $a-b\ge 2$ then $ab<(a-1)(b+1)$ and replacing such pairs $(a,b)$ of positive integers in the tuple with $(a-1,b+1)$ until the maximum is achieved. 

\cite[Conjecture 1.20]{ht-gr} is an analogue of \Cref{th:init} obtained by relaxing the constraints on the adjacency matrix of the graph to allow for non-negative {\it real} (rather than integral) entries. To state it we introduce, for a matrix
\begin{equation*}
  A\in M_n(\bR_{\ge 0})
\end{equation*}
with non-negative entries, the {\it weight}
\begin{equation*}
  |A|:=\sum_{i,j=1}^n A_{ij}
\end{equation*}
(i.e. the sum of all of its entries). Then, \cite[Conjecture 1.20]{ht-gr} reads

\begin{conjecture}\label{cj:init}
  Let $N$ be a non-negative real and $k$ a positive integer, and $A$ a nilpotent square matrix with entries in $\bR_{\ge 0}$ of weight $N$. Then,
  \begin{equation*}
    |A^k|\le \left(\frac Nk\right)^k
  \end{equation*}
  and equality is achievable. 
\end{conjecture}

A finite directed graph will provide a non-negative adjacency matrix $A$ as above, with rows and columns indexed by vertices and such that $A_{ij}$ is the number of edges from $i$ to $j$. The nilpotence encodes the fact that the graph is acyclic. 

\begin{remark}
\cite[Conjecture 1.20]{ht-gr} also imposes the condition that the directed graph underlying the matrix $A$ have no isolated vertices, i.e. that there be no $i$ such that the $i^{th}$ row and column are both zero. This condition seems unnecessary.  
\end{remark}

We can restate the conjecture by analogy to \Cref{cor:init}. 

\begin{conjecture}\label{cj:alt}
  Let $N$ be a non-negative real and $k$ a positive integer. The following numbers are then equal:
  \begin{itemize}
  \item the maximal weight of $A^k$, where $A$ is a nilpotent square matrix of weight $N$ with non-negative real entries;
  \item the maximal product of $k$ non-negative reals with sum $N$;
  \item $\left(\frac Nk\right)^k$. 
  \end{itemize}
\end{conjecture}

Of course, the fact that the last two items are equal is nothing but the arithmetic-geometric-mean inequality. We confirm \Cref{cj:init,cj:alt} as a particular case of one of the main results of the present note (see \Cref{th:mon,cor:getcj}): 

\begin{theorem}
\Cref{cj:alt} holds. 
  \qedhere
\end{theorem}

After a short introduction to the terminology and conventions in \Cref{se.prel} we prove \Cref{th:opt}, stating that given a closed additive submonoid $\cM$ of $\bR$ and a directed graph $E$, the supremum of
\begin{equation*}
  \sum_{\text{copies of $E$contained in $\Gamma$}} \text{product of labels of the edges of $E$}
\end{equation*}
as $\Gamma$ ranges over the $\cM$-labeled directed graphs equals the analogous supremum over only those $\Gamma$ for which every two edges lie on a common copy of $E\subset \Gamma$.

This then recovers \Cref{th:init}, proves \Cref{cj:alt}, and can be used to count the maximal number of $a$-arm stars in a directed graph with $N$ edges (\Cref{cor:1starcount}).  

\subsection*{Acknowledgements}

This work was partially supported by NSF grant DMS-1801011. 

I am grateful for input from P.M. Hajac and M. Tobolski.

\section{Preliminaries}\label{se.prel}

All graphs under discussion are finite and directed, so we often drop these adjectives. As in the Introduction, we abbreviate the phrase `directed acyclic graph' (i.e. one without oriented cycles of any length, including single-edge loops) as `DAG'.

\begin{definition}\label{def:lab}
  Let $\cM$ be a set with a distinguished symbol `$0$'. An {\it $\cM$-labeled} directed graph is a {\it simple} directed graph (i.e. no repeated edges) for which every pair $(x,y)$ of vertices carries a label $\ell(x,y)\in M$, with label $0$ precisely when $(x,y)$ is not and edge.

  Plain directed graphs, possibly with repeated edges, can be alternatively regarded as $\bZ_{\ge 0}$-labeled directed graphs {\it without} repeated edges, with $(x,y)$ carrying the label $m\in \bZ_{\ge 0}$ if there are $m$ edges $x\to y$.

  The {\it adjacency matrix} of an $\cM$-labeled graph on the vertex set $I$ is the $\cM$-valued matrix whose $(i,j)$ entry (for $i,j\in I$) is $\ell(x,y)$.  
\end{definition}

\begin{definition}\label{def:ct}
    If $\cM\subseteq \bR_{\ge 0}$ the {\it content} of an $\cM$-labeled graph $\Gamma$ is
  \begin{equation*}
    \mathrm{ct}~\Gamma:= \prod_{\text{edges }(x,y)} \ell(x,y)
  \end{equation*}
  and its {\it weight} is 
    \begin{equation*}
    \mathrm{wt}~\Gamma:= \sum_{\text{edges }(x,y)} \ell(x,y)
  \end{equation*}

  Similarly, if $S$ is a set of edges in $\Gamma$, the {\it $S$-exclusive content} of $\Gamma$ is
    \begin{equation*}
    \mathrm{ct}_{\not S}~\Gamma:= \prod_{\text{edges }(x,y)\not\in S} \ell(x,y). 
  \end{equation*}

\end{definition}

With all of this in place, \Cref{th:init,cor:init} (which in turn paraphrase \cite[Theorem 1.10 and Corollary 1.19]{ht-gr}) can be conjoined as

\begin{theorem}\label{th:init-all}
  Let $N$ and $k$ be two positive integers, and $N=kq+r$ be the decomposition of $N$ modulo $k$. The following quantities all admit the same maximal value $(q+1)^r q^{k-r}$.
  \begin{itemize}
  \item the number of length-$k$ paths in an $N$-edge DAG;
  \item the sum
    \begin{equation*}
      \sum_{\text{length-}k\text{ paths in }\Gamma}\mathrm{ct}(\text{path})
    \end{equation*}
    for $\bZ_{\ge 0}$-labeled DAGs $\Gamma$;
  \item the weight of $A^k$, where $A$ is a square $\bZ_{\ge 0}$-valued matrix of weight $N$;
  \item the product of $k$ non-negative integers with sum $N$. 
  \end{itemize}
  \qedhere
\end{theorem}

The fact that the first two optimization problems are identical follows immediately by recasting plain DAGs as labeled DAGs as in \Cref{def:lab}. On the other hand, translating this into the language of the third item is simply passing between a DAG and its adjacency matrix.

With this phrasing, \Cref{th:mon} below provides a direct generalization of \Cref{th:init-all}. 

\section{Optimizing labeled graphs}\label{se.main}

\Cref{th:init-all} and the results mentioned in the discussion preceding it are concerned with counting {\it paths} in a directed graph. We will first prove a general principle applicable to optimization problems of this general form, with general directed graphs in place of paths.

Specifically, let $E$ be a fixed finite {\it simple} directed graph (i.e. without repeated edges or loops), that will play the same role as a length-$k$ path did above.  Let also $(\cM,+)$ be a closed submonoid of $(\bR_{\ge 0},+)$.

\begin{definition}\label{def:sumct}
  Let $\Gamma$ be an $\cM$-labeled directed graph. We define
  \begin{equation*}
    \mathrm{ct}^E(\Gamma) = \mathrm{ct}_{\cM}^E(\Gamma):=\sum_{\alpha} \mathrm{ct}(\alpha), 
  \end{equation*}
  with $\alpha$ ranging over the subgraphs of $\Gamma$ isomorphic to $E$, with the $\cM$-labeling inherited from $\Gamma$.
\end{definition}

We then have

\begin{theorem}\label{th:opt}
  Let
  \begin{itemize}
  \item  $E$ be a simple directed graph;
  \item $(\cM,+)$ a closed submonoid of $(\bR_{\ge 0},+)$;
  \item $N\in \cM$ an element.   
  \end{itemize}
  Then, $\sup~ \mathrm{ct}^E_{\cM}(\Gamma)$ for $\cM$-labeled $\Gamma$ of weight $N$ is achieved over graphs $\Gamma$ with the following property:
  \begin{equation}\label{eq:all2}
    \text{Every two edges of }\Gamma\text{ belong to some common embedded copy }E\subset \Gamma. 
  \end{equation}
\end{theorem}
\begin{proof}
We have to prove that given an $\cM$-labeled $\Gamma$ of content $N$, $\mathrm{ct}^E_{\cM}$ can be improved by altering $\Gamma$ progressively until we achieve \Cref{eq:all2}. 

Indeed, suppose the edges $e$ and $f$ of $\Gamma$ do {\it not} belong to a common copy $E\subset \Gamma$. Then, the sets $\cS_e$ and $\cS_f$ of $E$-subgraphs of $\Gamma$ containing $e$ and $f$ respectively are disjoint.

  Now, for each path $\alpha\in \cS_e$ containing $e$, consider the $e$-exclusive content $\mathrm{ct}_{\not e}~\alpha$ as in \Cref{def:ct}, and similarly for $f$. Without loss of generality, suppose
  \begin{equation*}
    \Sigma_e:= \sum_{\alpha\in \cS_e}\mathrm{ct}_{\not e}~\alpha
  \end{equation*}
  is at least as large as its counterpart
  \begin{equation*}
    \Sigma_f:=\sum_{\beta\in \cS_f}\mathrm{ct}_{\not f}~\beta.
  \end{equation*}
  We can then eliminate edge $f$ and recycle its label into $e$, updating $\ell(e)$ to $\ell(e)+\ell(f)$. This modification of the graph will
  \begin{itemize}
  \item not decrease $\mathrm{ct}^E_{\cM}$; indeed, the latter is incremented by
    \begin{equation*}
      \ell(f)\left(\Sigma_e-\Sigma_f\right)\ge 0.
    \end{equation*}
  \item decrease the number of pairs of edges that do {\it not} belong to the same $E\subset \Gamma$. 
  \end{itemize}

  We can continue the process so long as there are such pairs of edges, so the procedure concludes precisely when we have obtained a graph satisfying \Cref{eq:all2}. This finishes the proof.
\end{proof}

\subsection{Paths}\label{subse:path}

\Cref{th:opt} has a number of consequences germane to the problems discussed in the introduction. The present subsection focuses on the case where the graph $E$ is a path, hence the relevance of the following simple observation. 

\begin{lemma}\label{le:allonpth}
  Let $k$ be a positive integer and $E$ a length-$k$ oriented path. Then, the only directed graphs $\Gamma$ satisfying \Cref{eq:all2} are length-$k$ paths and cycles any of the lengths $k+1$ up to $2k-1$.
  \qedhere
\end{lemma}

\begin{theorem}\label{th:mon}
  Let $k$ be a positive integer, $(\cM,+)$ a closed submonoid of $(\bR_{\ge 0},+)$ and $N\in \cM$ an element. The following quantities all admit the same maximal value.
  \begin{enumerate}[(1)]
  \item\label{item:pths} the sum
    \begin{equation*}
      \sum_{\text{length-}k\text{ paths in }\Gamma}\mathrm{ct}(\text{path})
    \end{equation*}
    for $\cM$-labeled DAGs $\Gamma$ of weight $N$;
  \item\label{item:mtrx} the weight of $A^k$, where $A$ is a square $\cM$-valued matrix of weight $N$;
  \item\label{item:nr} the product of $k$ non-negative elements of $\cM$ with sum $N$. 
  \end{enumerate}  
\end{theorem}
\begin{proof}
  The fact that \Cref{item:pths,item:mtrx} have the same optimal value follows by noting that if $A$ is the adjacency matrix of the labeled DAG $\Gamma$ then the length-$k$ paths in $\Gamma$ are in bijection with the non-zero entries of $A$, and those entries are precisely the contents of the respective paths.

  It thus remains to argue that the common maximal value of \Cref{item:pths,item:mtrx} also equals that of \Cref{item:nr}. This entails proving two inequalities:

  \begin{equation}\label{eq:31}
    \max~ \Cref{item:nr}\le \max ~ \Cref{item:pths}
  \end{equation}
and 
  \begin{equation}\label{eq:13}
    \max~ \Cref{item:pths}\le \max ~ \Cref{item:nr}.
  \end{equation}

  \Cref{eq:31} is easier to prove: simply note that every $k$-tuple of elements in $\cM$ can be realized as the $k$ labels of a length-$k$ path $\Gamma$.

  As for \Cref{eq:13}, \Cref{th:opt} applied to a $k$-path $E$ and \Cref{le:allonpth} imply that the maximum is achieved by an $\cM$-labeled length-$k$ path, and the labels of its $k$ edges will be the $k$ elements in \Cref{item:nr}.
\end{proof}

\begin{corollary}\label{cor:getcj}
\Cref{cj:alt} holds.   
\end{corollary}
\begin{proof}
  Simply take $\cM=\bR_{\ge 0}$ in \Cref{th:mon} and observe, as in the Introduction, that the maximal value in \Cref{item:nr} is achieved when all labels are equal to $\frac Nk$ by the arithmetic-geometric-mean inequality.
\end{proof}

\subsection{Stars}\label{subse:star}

The following notion of oriented tree is fairly common in the literature (see e.g. \cite[p.310]{kn}).

\begin{definition}\label{def:ortr}
  An {\it oriented tree with root $v$} (or {\it rooted at $v$}) is an oriented graph with a distinguished vertex $v$ such that for each vertex $w$ there is a unique oriented path $w\to v$.

  The {\it arms} of a rooted oriented tree are its maximal oriented paths (so they all connect a leaf to the root). 
\end{definition}

In this section we focus on specific classes of rooted directed trees.

\begin{definition}\label{def:star}
  A rooted directed tree is {\it $\ell$-equidistal} if all of its arms have the same length $\ell$. It is a {\it star} if any two arms intersect only at their common target (i.e. the root of the tree).

  Finally, a rooted directed tree is a {\it $\ell$-star} if it is both a star and $\ell$-equidistal. 
\end{definition}

The preceding discussion focused on $k$-paths, which are $k$-stars with one arm. At the other end of the spectrum, we can consider $1$-stars with $a$ arms instead. The analogue of \Cref{th:mon} is

\begin{theorem}\label{th:1star}
  Let
  \begin{itemize}
  \item $E$ be a $1$-star with $a$ arms;
  \item $(\cM,+)$ a closed submonoid of $(\bR_{\ge 0},+)$;
  \item $N\in \cM$.
  \end{itemize}
  The following quantities all admit the same supremum. 
  \begin{enumerate}[(1)]
  \item\label{item:1stars} the sum
    \begin{equation*}
      \mathrm{ct}_{\cM}^{E}(\Gamma) = \sum_{\text{$1$-stars with $a$ arms contained in $\Gamma$}}\mathrm{ct}(star)
    \end{equation*}
    for $\cM$-labeled DAGs $\Gamma$ of weight $N$;
  \item\label{item:elem} the $a^{th}$ elementary symmetric sum evaluated at some $t$-tuple of elements in $\cM$ with sum $N$ (for varying $t$):
    \begin{equation}\label{eq:elemsymlambdas}
      \sum_{1\le i_1<\cdots<i_a\le t}\lambda_{i_1}\cdots\lambda_{i_a},\quad \lambda_i\in \cM,\quad \sum\lambda_i=N.
    \end{equation}
  \end{enumerate}    
\end{theorem}
\begin{proof}
  According to \Cref{th:opt} it is enough to range over $\cM$-labeled DAGs $\Gamma$ for which every two edges lie in some common copy of $E\subset \Gamma$. This clearly implies that $\Gamma$ itself must be a $1$-star, with, say, $t$ arms.

  If $\lambda_i$, $1\le i\le t$ are the labels of the $t$ arms of $\Gamma$ so that
  \begin{equation*}
    \sum_{i=1}^t \lambda_i = N,
  \end{equation*}
  then the content $\mathrm{ct}^E(\Gamma)$ is the $a^{th}$ elementary symmetric function evaluated at the $\lambda_i$. This concludes the proof.
\end{proof}

The following consequence is a kind of continuous version of counting the maximal number such stars in a DAG with $N$ edges.

\begin{corollary}\label{cor:1starr}
  Let
  \begin{itemize}
  \item $E$ be a $1$-star with $a$ arms;
  \item $N\in \bR_{\ge 0}$.
  \end{itemize}
  The supremum
  \begin{equation*}
    \sup_{\Gamma}~ \mathrm{ct}_{\bR_{\ge 0}}^E(\Gamma),\quad \text{$\Gamma$ an $\bR_{\ge 0}$-labeled directed graph of weight $N$}
  \end{equation*}
  is $\frac{N^a}{a!}$.
\end{corollary}
\begin{proof}
  According to \Cref{th:1star}, we want the supremum of \Cref{eq:elemsymlambdas} for $\lambda_i\in \bR_{\ge 0}$ and varying $t$. For fixed $t$ that expression is maximal when all $\lambda_i$ are equal (to $\frac Nt$), so \Cref{eq:elemsymlambdas} is at most  
  \begin{equation*}
    \tbinom{t}{a}\cdot\left(\frac{N}{t}\right)^a = \frac{N^a t(t-1)\cdots (t-a+1)}{a! t^a}. 
  \end{equation*}
  As $t\to\infty$ the right hand side converges to its supremum $\frac{N^a}{a!}$, hence the conclusion.
\end{proof}

As for the discrete version, it reads

\begin{corollary}\label{cor:1starcount}
  Let $N$ and $a$ be two positive integers. The following quantities all have the same maximal value $\tbinom{N}{a}$
  \begin{enumerate}[(1)]
  \item\label{item:1} the number of $1$-stars with $a$ arms contained in directed graph with $N$ edges;
  \item\label{item:2} the sum
    \begin{equation*}
      \sum_{\text{$1$-stars with $a$ arms contained in $\Gamma$}}\mathrm{ct}(star)
    \end{equation*}
    for $\bN_{\ge 0}$-labeled directed graphs $\Gamma$ of weight $N$;
  \item\label{item:3} the $a^{th}$ elementary symmetric sum evaluated at some $t$-tuple of non-negative integers with sum $N$ (for varying $t$):
    \begin{equation}\label{eq:lambdasn}
      \sum_{1\le i_1<\cdots<i_a\le t}\lambda_{i_1}\cdots\lambda_{i_a},\quad \lambda_i\in \bN_{\ge 0},\quad \sum\lambda_i=N.
    \end{equation}
  \end{enumerate}    
\end{corollary}
\begin{proof}
That \Cref{item:1} and \Cref{item:2} have the same optimal value follows as in the discussion following \Cref{th:init-all}, by recasting repeated edges in a directed graph as $\bN_{\ge 0}$-labels. On the other hand, the fact that \Cref{item:2,item:3} have the same optimal value follows from \Cref{th:1star} applied to $\cM=\bN_{\ge 0}$ and $E$ an $a$-arm $1$-star. It thus remains to prove that the supremum is a maximum, and that that maximum is $\tbinom{N}{a}$.  

As in the proof of \Cref{th:1star}, we can assume $\Gamma$ is a $1$-star with $t$ arms and respective labels $\lambda_i\in \bN_{>0}$, $1\le i\le t$ (the labels can be assumed positive because $0$ labels make no contribution to \Cref{eq:lambdasn}). In particular, $t\le N$. 

Having fixed $t$, we observed in the proof of \Cref{cor:1starr} that the elementary symmetric function \Cref{eq:lambdasn} is dominated by
\begin{equation*}
  \tbinom{t}{a}\cdot\left(\frac{N}{t}\right)^a = \frac{N^a t(t-1)\cdots (t-a+1)}{a! t^a}.
\end{equation*}
That expression is increasing in $t$, so reaches its maximum at $t=N$. That maximum is precisely
\begin{equation*}
  \tbinom{t}{a} = \tbinom{N}{a},
\end{equation*}
and is achievable by an $\bN_{\ge 0}$-labeled $N$-armed $1$-star by simply assigning label $\lambda_i=1$ to each of the $N$ edges.
\end{proof}


\bibliographystyle{plain}

\Addresses

\end{document}